\documentclass[12pt]{elsarticle}
\usepackage{amsmath,amstext,amssymb,amsthm,enumerate,bbm}
\usepackage[T2A]{fontenc}
\usepackage[cp1251]{inputenc}
\usepackage[english]{babel}
\usepackage[margin=2cm]{geometry}

\usepackage{graphicx}        % standard LaTeX graphics tool
\usepackage{multicol}        % used for the two-column index
\newtheorem{theorem}{Theorem}[section]
\newtheorem{proposition}{Proposition}[section]
\newtheorem{lemma}{Lemma}[section]
\newtheorem{corollary}{Corollary}[section]
\theoremstyle{definition}

\newtheorem{remark}{Remark}[section]

\DeclareMathOperator{\pr}{\mathsf P}

\newcommand{\eps}{\varepsilon}

\newcommand{\R}{\mathbb{R}}

\newcommand{\F}{\mathcal{F}}

\newcommand{\CC}{\mathcal C}
\newcommand{\abs}[1]{\left\vert#1\right\vert}
\newcommand{\norm}[1]{\left\lVert#1\right\rVert}
\newcommand{\norb}[1]{\big\lVert#1\big\rVert}
\newcommand{\set}[1]{\left\{#1\right\}}
\newcommand{\ex}[1]{\mathsf{E}\left[#1\right]}
\newcommand{\exo}[1]{\mathsf{E}_0\left[#1\right]}
\newcommand{\Ex}[1]{\mathsf{E}\Big[#1\Big]}
\newcommand{\ind}[1]{\mathbbm{1}_{#1}}

\def\ssf(#1){\spacefactor=#1\relax}

\begin{document}
\title{Convergence of solutions of mixed stochastic delay differential equations with applications}
\author[u1]{Yuliya Mishura}
\ead{myus@univ.kiev.ua}
%\address{Kyiv National Taras Shevchenko University, Department of Mechanics and Mathematics, Volodymyrska str. 64, 01601, Kyiv, Ukraine}
\author[u1,u2]{Taras Shalaiko}
\ead{tarasenya@gmail.com}
\author[u1]{Georgiy Shevchenko\corref{cor1}}
\cortext[cor1]{Corresponding author}
\ead{zhora@univ.kiev.ua}
\address[u1]{Taras Shevchenko National University of Kyiv,  Mechanics and Mathematics Faculty, Volodymyrska 64, 01601 Kyiv, Ukraine}
\address[u2]{Mannheim University, Institute of Mathematics, Department of Mathematical Economics II, A5, 6, D-68199, Mannheim, Germany}

\begin{abstract}
The paper is concerned with a mixed stochastic delay  differential equation involving both a Wiener process and a $\gamma$-H\"older continuous process with $\gamma>1/2$ (e.g.\ a fractional Brownian motion with Hurst parameter greater than $1/2$).
It is shown that its solution depends continuously on the coefficients and the initial data. Two applications of this result are given: the convergence of solutions to equations with vanishing delay to the solution of equation without delay and the convergence of Euler approximations for mixed stochastic differential equations. As a side result of independent interest, the integrability of solution to mixed stochastic delay differential equations is established.
\end{abstract}
\begin{keyword}
%% keywords here, in the form: keyword \sep keyword
Mixed stochastic differential equation\sep  stochastic delay differential equation\sep convergence of solutions \sep fractional Brownian motion \sep vanishing delay \sep Euler approximation 

\MSC[2010] 60H10 \sep 60H07 \sep  60G22
%% MSC codes here, in the form: \MSC code \sep code
%% or \MSC[2008] code \sep code (2000 is the default)

\end{keyword}

\maketitle

\section*{Introduction}
In this paper we consider a multidimensional mixed stochastic delay differential equation
 \begin{equation}\label{sdde-intro}
%\begin{aligned}
d X(t) = a(t,X_t)dt+
b(t,X_t)dW(t)+ c(t,X_t)dZ(t),\ t\in[0,T].
%\end{aligned}
 \end{equation}
In this equation the coefficients $a,b,c$ depend on the past $X_s = \set{X({s+u}),u\in[-r,0]}$ of the process $X$, and the initial condition is thus given by $X(t)=\eta(t)$, $t\in[-r,0]$, where $\eta\colon [-r,0]\to \R$ is some function. Equation \eqref{sdde-intro} is driven by two random processes: a standard Wiener process  $W$  and a process $Z$, whose trajectories are $\gamma$-H\"older continuous with $\gamma>1/2$. The process $Z$ is usually a long memory process, e.g.\ a fractional Brownian motion $B^H$ with the Hurst parameter $H>1/2$.

It\^o stochastic delay differential equations, i.e.\ those with $c=0$, were investigated in many articles, see  \cite{mao, mohammed} and references therein. Fractional stochastic delay differential equations, in which $b=0$ and $Z = B^H$, were considered in only few papers. 
For $H>1/2$, the existence and uniqueness of solution under different sets of assumptions was established in \cite{besalu-rovira,boufoussi-hajji1,caraballo,ferrante-rovira,ferrante-rovira1, leon-tindel}. In the case $H>1/3$, the existence and uniqueness of solution was shown in \cite{tindel-neuenkirch-nourdin} for coefficients of the form $f(X(t),X(t-r_1),X(t-r_2),\dots)$. 

The existence and uniqueness of solution to equation \eqref{sdde-intro} was established in \cite{mbfbm-delay}, where also the finiteness of moments was shown under the additional assumption that the coefficient $b$ is bounded.  Mixed equations without delay were considered in articles \cite{guerra-nualart,kubilius,mbfbm-sde,mbfbm-limit,mbfbm-integr,mbfbm-malliavin}.

In this article we prove  that if the coefficients and initial conditions of mixed stochastic delay  differential equations converge, then their solutions converge uniformly in probability.  We give two applications of this result. First we show that when the delay vanishes, solutions of mixed stochastic delay  differential equations converge uniformly in probability to a solution of equation without delay. Then we establish the uniform convergence of Euler approximations for mixed stochastic differential equation towards its solution. We also extend the results of \cite{mbfbm-integr} about the integrability of  solution to \eqref{sdde-intro}, dropping the assumption that $b$ is bounded.

The paper is organized as follows. Section~\ref{sec:prelim} provides necessary information about pathwise stochastic integration and describes the notation used in the article. Section~\ref{sec:limit} contains the main result of the article about convergence of solutions to mixed stochastic delay differential equations with convergent coefficients and initial conditions. It also contains a result of independent interest about the integrability of solution to mixed stochastic delay differential equation. Section~\ref{sec:applic} is devoted to applications of the convergence results. In Subsection~\ref{subsec:vanish}, a sequence of mixed stochastic delay differential equations is considered with delay horizon converging to zero, and it is shown that their solutions converge to a solution of equation without delay. In Subsection~\ref{subsec:euler} it is proved that the Euler approximations for  mixed stochastic differential equation (without delay) converge to its solution, as the mesh of partition goes to zero. Appendix contains an auxiliary result about convergence of solutions to It\^o stochastic delay differential equations with random coefficients, which is used in the proof of main theorem.

\section{Preliminaries}\label{sec:prelim}

Let $(\Omega,\mathcal{F},\mathbb{F} = \{\mathcal{F}_t,t\ge 0\}, P)$ be a standard stochastic basis. 

Throughout the article, $\abs{\cdot}$ will denote the absolute value of a real number, the Euclidean norm of a vector, or the operator norm of a matrix. The symbol $C$ will denote a generic constant, whose
value may change from one line to another. To emphasize its dependence on some parameters, we will put them into subscripts.

We will need the notion of generalized fractional Lebesgue--Stieltjes integral. Below we give only basic information on the integral, further details and proofs can be found in \cite{zahle}. %It is defined as follows.

Let  $f,g\colon[a,b]\to \R$, $\alpha\in (0,1)$. Define the forward and backward fractional Riemann--Liouville derivatives
\begin{gather*}
\big(D_{a+}^{\alpha}f\big)(x)=\frac{1}{\Gamma(1-\alpha)}\bigg(\frac{f(x)}{(x-a)^\alpha}+\alpha
\int_{a}^x\frac{f(x)-f(u)}{(x-u)^{1+\alpha}}du\bigg),\\
\big(D_{b-}^{1-\alpha}g\big)(x)=\frac{e^{i\pi
\alpha}}{\Gamma(\alpha)}\bigg(\frac{g(x)}{(b-x)^{1-\alpha}}+(1-\alpha)
\int_{x}^b\frac{g(x)-g(u)}{(u-x)^{2-\alpha}}du\bigg),
\end{gather*}
where $x\in(a,b)$.

The generalized Lebesgue--Stieltjes
integral %$\int_a^bf(x)dg(x)$
is defined as
\begin{equation*}\int_a^bf(x)dg(x)=e^{i\pi\alpha}\int_a^b\big(D_{a+}^{\alpha}f\big)(x)\big(D_{b-}^{1-\alpha}g_{b-}\big)(x)dx
\end{equation*}
provided the integral in the right-hand side exists. 
For functions $f,g\colon [a,b]\to \mathbb \R$ and a number $\alpha\in(0,1)$ define
\begin{gather}\label{norm1a}
\norm{f}_{1,\alpha;[a,b]} = \int_a^b \left(\frac{\abs{f(a)}}{(t-a)^\alpha}
+ \int_a^t \frac{\abs{f(t)-f(s)}}{(t-s)^{1+\alpha}}ds\right)dt,\\
\norm{g}_{0,\alpha;[a,b]} = \sup_{a\le s\le t\le b}\left(\frac{\abs{g(t) -g(s)}}{(t-s)^{1-\alpha}}
+ \int_s^t \frac{\abs{g(u)-g(s)}}{(u-s)^{2-\alpha}}du\right).
\end{gather}
Note that only the first expression defines a norm, the second defines a seminorm. 
The following fact is evident: if $\norm{f}_{1,\alpha;[a,b]}< \infty$
and $\norm{g}_{0,\alpha;[a,b]}< \infty$, then the generalized Lebesgue--Stieltjes integral is well defined and admits an estimate 
\begin{equation}\label{integrfbmestimate}
\abs{\int_a^b f(x) dg(x)}\le \frac{1}{\Gamma(\alpha)\Gamma(1-\alpha)}\norm{f}_{1,\alpha;[a,b]}\norm{g}_{0,\alpha;[a,b]}.
\end{equation}
We will also need an estimate in terms of H\"older norms. 
Namely, if $f\in C^\lambda[a,b]$, $g\in C^\mu[a,b]$ with $\lambda\in(0,1),\mu\in(0,1)$ and $\lambda+\mu>1$, then the generalized Lebesgue--Stieltjes integral is well defined and coincides with the limit of Riemann--Stieltjes integral sums. Moreover, the Young--Love inequality holds:
\begin{equation*}
\left|\int_{a}^b f(s) dg(s)\right|\leq C_{\lambda,\mu} \norm{g}_{a,b,\mu}\big(\norm{f}_{a,b,\infty}+\norm{f}_{a,b,\lambda}(b-a)^{\lambda}\big)(b-a)^{\mu},
%\label{integrestim}
\end{equation*}
where $\norm{f}_{a,b,\infty} = \sup_{x\in[a,b]}\abs{f(x)}$
is the supremum norm on $[a,b]$, and
$$
\norm{f}_{a,b,\lambda} = \sup_{a\le x<y\le b}\frac{\abs{f(y)-f(x)}}{(y-x)^\lambda}
$$
is the H\"older seminorm.

\section{Limit theorem for mixed equation}\label{sec:limit}

Fix some $T>0$ and $r>0$ (they will play a role of time horizon and delay horizon, respectively) and denote by $\CC = C([-r,0];\R^d)$ the space of continuous $\R^d$-valued functions defined on the interval $[-r,0]$. This space is a Banach space with the supremum norm $\norm{\psi}_{\CC} = \max_{s\in[-r,0]}\abs{\psi(s)}$, $\psi\in\CC$. 

In order to introduce the dependence on past, for a stochastic process $\xi = \set{\xi(t),t\in[-r,T]}$ define a \textit{segment} $\xi_t\in \CC$ at the point $t\in[0,T]$ by $\xi_s(u) = \xi(s+u)$, $u\in[-r,0]$. 

Consider the following sequence of mixed stochastic delay differential equations (SDDEs) in $\R^d$ indexed by $n\ge 0$:
\begin{equation}\label{sdde-coord}
\begin{aligned}
& X^n(t) = \eta^n(0) + \int_{0}^{t} a^n(s,X^n_s)ds + \sum_{i=1}^{m} \int_0^t b^n_i(s,X^n_s)dW_i(s) + \sum_{j=1}^{l}\int_0^t c^n_j(s,X^n_s)dZ_j(s), t\in [0,T],\\
& X^n(t) = \eta^n(t), t\in[-r,0].
\end{aligned}
\end{equation}
Here $a^n,b^n_i,c^n_j\colon  [0,T]\times \CC\to\R^d$, $i=1,\dots,m$,  $j=1,\dots,l$, are measurable functions; $Z = \set{Z(t),t\in[0,T]}$ is an $\mathbb{F}$-adapted process in $\R^l$ such that its trajectories are almost surely H\"older continuous of order $\gamma>1/2$; $W = \set{W(t),t\in[0,T]}$ is an $\mathbb{F}$-Wiener process in $\R^m$, the ``initial condition'' $\eta\colon[0,T]\to \R^d$ is non-random. 
The integral w.r.t.\ the Wiener process is understood as the usual It\^o integral, which is well defined provided that the integrand belongs to $L^2[0,T]$ a.s. The integral w.r.t.\ is understood in generalized Lebesgue--Stieltjes sense. 

We will assume the following about the coefficients of equations \eqref{sdde-coord}:
\begin{enumerate}[H1.]
\item
For all $\psi\in\CC$, $t\in[0,T]$,
\begin{gather*}
|a^n(t,\psi)|+|b^n(t,\psi)|+|c^n(t,\psi)|\leq K(1+\norm{\psi}_{\CC}).
\end{gather*}
\item
For all $t\in[0,T]$, $\psi\in\CC$,  $c^n$ has a Fr\'echet derivative  $\partial_\psi c^n(t,\psi)$ belonging to the space $L(\CC, \R^d)$ of bounded linear operators from $\CC$ to $\R^d$, and this derivative is bounded uniformly in $t\in[0,T],\psi\in\CC$:
$$
\norm{\partial_{\psi}c^n(t,\psi)}_{L(\CC, \R^d)}\le K.
$$
\item The functions $a^n$, $b^n$ and $\partial_\psi c^n$ are locally Lipschitz continuous in $\psi$: for any $R>0$, $t\in[0,T]$, and all $\psi_1,\psi_2\in \CC$ with $\norm{\psi_1}_\CC\le R$,  $\norm{\psi_2}_\CC\le R$,
$$|a^n(t,\psi_1)-a^n(t,\psi_2)|+|b^n(t,\psi_1)-b^n(t,\psi_2)|+\norm{\partial_\psi c^n(t,\psi_1)-\partial_\psi c^n(t,\psi_2)}_{L(\CC, \R^d)}\leq K_R\norm{\psi_1-\psi_2}_{\CC}.$$
\item The functions $c^n$ and $\partial_\psi c^n$ are H\"older continuous in $t$: for some  $\beta\in(1-\gamma,1)$ and for all $t_1,t_2\in[0,T]$, $\psi\in\CC$
$$|c^n(t_1,\psi)-c^n(t_2,\psi)|\le K|t_1-t_2|^\beta(1+\norm{\psi}_\CC),\quad \norm{\partial_\psi c^n(t_1,\psi)-\partial_\psi c^n(t_2,\psi)}_{L(\CC, \R^d)}\leq K|t_1-t_2|^\beta.$$
\item The initial condition  $\eta^n$ is a H\"older continuous function: for some  $\theta\in(1-\gamma, 1/2)$ and for all $t_1,t_2\in[0,T]$
$$
\abs{\eta^n(t_1)-\eta^n(t_2)}\le K\abs{t_1-t_2}^\theta.
$$
\end{enumerate}

Fix some  $\alpha\in(1-\gamma,1/2)$, denote $h(t,s)=(t-s)^{-1-\alpha}$ and define   for an either real- or vector-valued function $f$
\begin{align*}
%\norm{f}_{0;t} &= \sup_{0\le u<v\le t} \left(\frac{\abs{f(v)-f(u)}}{(v-u)^{1-\alpha}} + \int_u^v \frac{\abs{f(u)-f(z)}}{(z-u)^{2-\alpha}}dz\right),\\
\norm{f}_{\infty,t} &= \sup_{s\in[-r,t]}\abs{f(s)},\\
\norm{f}_{1,t} &= \int_0^t \norm{f({\cdot+t-s})-f(\cdot)}_{\infty,s} h(t,s)ds,\\
\norm{f}_{t} &= \norm{f}_{\infty,t}+\norm{f}_{1,t}.
\end{align*}
Also denote for brevity $\norm{f}_{0;t} = \norm{f}_{0,\alpha;[0,t]}$.
It is proved in \cite[Theorem 4.1]{mbfbm-delay} that under the assumptions H1--H5 equation \eqref{sdde-coord} has a unique solution, which is an $\mathbb{F}$-adapted process $X$ such that
$\norm{X}_T<\infty$ a.s., and \eqref{sdde-coord} holds almost surely for all $t\in[0,T]$. 

In the following we will abbreviate equations \eqref{sdde-coord} and their ingredients as
\begin{equation}\label{main-sdde}
X^n(t) = \eta^n(0) + \int_{0}^{t} a^n(s,X^n_s)ds +  \int_0^t b^n(s,X^n_s)dW(s) + \int_0^t c^n(s,X^n_s)dZ(s).
\end{equation}

We will need some auxiliary results from \cite{mbfbm-delay}, which we for convenience state here in a slightly modified form.

\begin{lemma}\label{prop-apriormoment}
Let the coefficients of equation 
\begin{align*}
Y(t) &= \eta(0) + \int_{0}^{t} a(s,Y_s)ds +  \int_0^t b(s,Y_s)dW(s) + \int_0^t c(s,Y_s)dZ(s),\ t\in[0,T],\\
Y(t)& = \eta(t),\ t\in[-r,0]
%\label{eqY}
\end{align*}
satisfy \textup{H1--H5}, and let $A_{M} = \set{\norm{Z}_{0;T}\le M}$ for $M\ge 1$. Then for each  $p\ge 1$ there is a constant $C=C_{M,p,r,T,K,\alpha,\eta(0)}$ such that 
$$
\ex{\norm{Y}_{T}^p\ind{A_{M}}}\le C.
$$
\end{lemma}

\begin{lemma}\label{prop-quasicontract}
Let the coefficients of equations
\begin{align*}
Y^{i}(t)& = \eta(0) + \int_{0}^{t} a(s,Y^i_s)ds +  \int_0^t b(s,Y^i_s)dW(s) + \int_0^t c(s,Y^i_s)dZ^i(s),\quad t\in[0,T], i=1,2,\\
Y^{i}(t)& = \eta(t),\ t\in[-r,0]
\end{align*}
satisfy \textup{H1--H5}, and let $A_{M,R} = \set{\norm{Z^1}_{0;T}\le M,\norm{Z^2}_{0;T}\le M,\norm{Y^1}_T\le R, \norm{Y^2}_T\le R}$ for  $M\ge 1$, $R\ge 1$. Then for each $p\ge 4/(1-2\alpha)$ there is a constant $C = C_{M,R,p,r,K,K_R,\alpha}$ such that
$$
\ex{\norm{Y^1-Y^2}_{\infty,T}^p\ind{A_{M,R}}}\le C\ex{\norm{Z^1-Z^2}^p_{0;T}\ind{A_{M,R}}}.
$$
\end{lemma}
We will also need a result about the finiteness of moments  of $\norm{Y}_{0,T,\infty}$.
 \begin{theorem}
  \label{thm:moments} 
 Let the coefficients of equation 
 \begin{align*}
 Y(t) &= \eta(0) + \int_{0}^{t} a(s,Y_s)ds +  \int_0^t b(s,Y_s)dW(s) + \int_0^t c(s,Y_s)dZ(s),\ t\in[0,T],\\
 Y(t)& = \eta(t),\ t\in[-r,0]
 %\label{eqY}
 \end{align*}
 satisfy \textup{H1--H5}, and the process $\{Z(t),t\in[0,T]\}$ be such that
 $$\ex{\exp\{\norm{Z}^{1/\gamma}_{0,T,\gamma}\}}\le L.$$
 Then for any $p>0$ there is a constant $C = C_{p,r,T,K,L,\eta(0)}$ such that 
 $$\ex{\norm{Y}^p_{0,T,\infty}}\le C.$$
 \end{theorem}
 \begin{proof}
This proof closely follows that of {Theorem 1} in \cite{mbfbm-integr}. Throughout the proof, $C$  will denote a generic constant depending on $p,r,T,K,L,\eta(0)$.

Fix some  $N \ge 1$, $R\ge 1$ and denote $\tau_{N,R} = \min\set{t\ge 0: \norm{Y}_{0,t,\infty}\ge R\text{ or } \norm{Z}_{0,t,\gamma}\ge N}$, $Y^{N,R}(t) = Y({t\wedge \tau_{N,R}})$, $\ind{t} = \ind{\set{t\le \tau_{N,R}}}$. Put
$I^a_t = \int_0^t a(s,Y^{N,R}_s)\ind{s}ds$, $I^b_t = \int_0^t b(s,Y^{N,R}_s)\ind{s}dW(s)$, $I^c_t = \int_0^t c(s,Y^{N,R}_s)\ind{s}dZ(s)$. %Assume without loss of generality that  $\theta< 1/2$.

Take a number $\Delta\in(0,1)$, whose value will be specified later. For $t \in [0,T]$ and $u,v$ such that $0\le u < v\le u+\Delta \le t$ write
\begin{align*}
\abs{Y^{N,R}(v)-Y^{N,R}(u)}&\le \abs{I^a_v-I^a_u} +\big|{I^b_v-I^b_u}\big|+ \abs{I^c_v-I^c_u}. \end{align*}
Estimate first
\begin{gather*}
\abs{I^a_v-I^a_u} \le \int_u^v \abs{a(z,Y^{N,R}_z)} dz \le C\int_u^v \left(1+\norm{Y^{N,R}_z}_{\CC}\right) dz \\
\le C\left(1+\norm{Y^{N,R}}_{0,t,\infty} + \norm{\eta}_{\CC}\right)(v-u)
\le C\left(1+\norm{Y^{N,R}}_{0,t,\infty} \right)(v-u),
\end{gather*}
where in the last step we have used the following simple observation: $\norm{\eta}_{\CC}\le \abs{\eta(0)} + r^\theta \norm{\eta}_{-r,0,\theta}\le C + K(r+1)\le C$.

From the Young--Love inequality it follows that
\begin{align*} \abs{I^c_v-I^c_u} \le C N  \left(
\norm{c(\cdot,Y^{N,R}_\cdot)}_{u,v,\infty} + 
\norm{c(\cdot,Y^{N,R}_\cdot)}_{u,v,\theta}(v-u)^\theta\right)(v-u)^\gamma.
\end{align*}
Further, from the linear growth assumption $$\norm{c(\cdot,Y^{N,R}_\cdot)}_{u,v,\infty}\le C \left(1+\norm{Y^{N,R}}_{0,v,\infty} + \norm{\eta}_\CC\right)\le C \left(1+\norm{Y^{N,R}}_{0,t,\infty}\right).$$ 
Define %for $h>0$
$$
\norm{f}_{a,b;\Delta,\theta} = \sup_{\substack{a\le x<y\le b,\\
y-x\le \Delta} } \frac{\abs{f(y)-f(x)}}{(y-x)^\theta}.
$$
Since
\begin{align*}
\abs{c(x,Y^{N,R}_x)-c(y,Y^{N,R}_y)}&\le \abs{c(x,Y^{N,R}_x)-c(y,Y^{N,R}_x)} + \abs{c(y,Y^{N,R}_x)-c(y,Y^{N,R}_y)} \\ &\le C\left(\abs{x-y}^\beta\big(1+\norm{Y^{N,R}_x}_{\CC}\big) + \norm{Y^{N,R}_x - Y^{N,R}_y}_{\CC} \right),
\end{align*}
then
\begin{gather*}
\norm{c(\cdot,Y^{N,R}_\cdot)}_{u,v,\theta}\le C\left((v-u)^{\beta-\theta}\big(1+\norm{Y^{N,R}}_{0,v,\infty}+\norm{\eta}_{\CC}\big) + \norm{Y^{N,R}}_{0,t;\Delta,\theta}+ \norm{\eta}_{-r,0,\theta}\right)\\
\le C\left(1 + (v-u)^{\beta-\theta}\big(1+\norm{Y^{N,R}}_{0,t,\infty}\big) + \norm{Y^{N,R}}_{0,t;\Delta,\theta}  \right).
\end{gather*}
Therefore, 
\begin{align*}
\abs{I^c_v-I^c_u}
\le CN \left(1+\norm{Y^{N,R}}_{0,t,\infty} + \norm{Y^{N,R}}_{0,t;\Delta,\theta} (v-u)^\theta \right)(v-u)^\gamma.
\end{align*}
Collecting the above estimates, we get
\begin{align*}
\norm{Y^{N,R}}_{0,t;\Delta,\theta}&\le C\left(1+\norm{Y^{N,R}}_{0,t,\infty}\right)\Delta^{1-\theta}
+ \norb{I^b}_{0,t;\Delta,\theta}\\
&\qquad  + CN  \left(1+\norm{Y^{N,R}}_{0,t,\infty}\Delta^{\gamma-\theta} + \norm{Y^{N,R}}_{0,t;\Delta,\theta} \Delta^\gamma \right)\\
&\le \norb{I^b}_{0,t;\Delta,\theta} + K' N\left(1+\norm{Y^{N,R}}_{0,t,\infty}\Delta^{\gamma-\theta} + \norm{Y^{N,R}}_{0,t;\Delta,\theta} \Delta^\gamma \right)
\end{align*}
with certain non-random constant $K' $. 

Suppose that $\Delta \le (2K'  N)^{-1/\gamma}$. Then
\begin{equation}\label{xnr}
\norm{Y^{N,R}}_{0,t;\Delta,\theta}\le 2\norb{I^b}_{0,t;\Delta,\theta} + 2K' N\left(1+\norm{Y^{N,R}}_{0,t,\infty}\Delta^{\gamma-\theta} \right).
\end{equation}
Let $s\in[0\vee(t-\Delta),t]$. Then from the obvious inequality
$$\norm{Y^{N,R}}_{0,t,\infty}\le \norm{Y^{N,R}}_{0,s,\infty} + \norm{Y^{N,R}}_{s,t,\theta} (t-s)^\theta\le \norm{Y^{N,R}}_{0,s,\infty} + \norm{Y^{N,R}}_{0,t;\Delta,\theta} \Delta^\theta
$$
using \eqref{xnr}, we obtain
\begin{align*}
\norm{Y^{N,R}}_{0,t,\infty}&\le \norm{Y^{N,R}}_{0,s,\infty} +  2\left(\norb{I^b}_{0,t;\Delta,\theta} + K' N\right)(t-s)^\theta+2K'N\norm{Y^{N,R}}_{0,t,\infty} (t-s)^\gamma\\
&\le \norm{Y^{N,R}}_{0,s,\infty} +  2\left(\norb{I^b}_{0,t;\Delta,\theta} + K' N\right)\Delta^\theta+2K'N\norm{Y^{N,R}}_{0,t,\infty} \Delta^\gamma.
\end{align*}
Assuming further that $\Delta\le (4K'  N)^{-1/\gamma}$, we get
$$
\norm{Y^{N,R}}_{0,t,\infty}\le 2\norm{Y^{N,R}}_{0,s,\infty} +  4\left(\norb{I^b}_{0,t;\Delta,\theta}+K'N\right)\Delta^\theta.
$$
Hence we derive for any $p>1$ that
\begin{equation}\label{exxnrp}
\ex{\norm{Y^{N,R}}^p_{0,t,\infty}}\le C \left(\ex{\norm{Y^{N,R}}_{0,s,\infty}^p} + \ex{\norb{I^b}_{0,t;\Delta,\theta}^p}\Delta^{p\theta}+N^p\right).
\end{equation}
Take some $\kappa\in(\theta,1/2)$.
Obviously, $\norb{I^b}_{0,t;\Delta,\theta}\le \Delta^{\kappa-\theta}\norb{I^b}_{0,t,\kappa}$. Assuming that $p>(1/2-\kappa)^{-1}$ and using the Garsia--Rodemich--Rumsey inequality, we get
\begin{align*}
\ex{\norb{I_b}_{0,t,\kappa}^p}&\le C \int_0^t\int_0^t\frac{\ex{\abs{I^b(x)-I^b(y)}^p}}{\abs{x-y}^{p\kappa +2}}dx\,dy\\
& \le C \int_0^t \int_0^t \ex{\abs{\int_x^y \abs{b(z,Y^{N,R}_z)}^2\ind{z} dz}^{p/2}}{\abs{x-y}^{-p\kappa -2}}dx\,dy\\
& \le C \int_0^t\int_0^t \int_x^y \left(1+\ex{\norm{Y^{N,R}_z}_{\CC}^{p}}\right) dz \abs{x-y}^{p/2-p\kappa -3}dx\,dy\\
&\le C\left(1+ \ex{\norm{Y^{N,R}}_{0,t,\infty}^p}+\norm{\eta}_{\CC}^p\right)\int_0^t \int_0^t \abs{x-y}^{p/2-p\kappa-2} dx\,dy\\
&\le  C \left(1+ \ex{\norm{Y^{N,R}}_{0,t,\infty}^p}\right).
\end{align*}
Plugging this estimate into \eqref{exxnrp}, we arrive at the inequality
\begin{align*}
\ex{\norm{Y^{N,R}}_{0,t,\infty}^p}
%&\le K'_{p} \left(\ex{\norm{Y^{N,R}}_{0,s,\infty}^p} + \ex{\norm{Y^{N,R}}_{0,t,\infty}^p}\Delta^{p/2}+N^p\Delta^{p\theta}\right)\\
%&
\le K'_p\left(\ex{\norm{Y^{N,R}}_{0,s,\infty}^p} + \ex{\norm{Y^{N,R}}_{0,t,\infty}^p}\Delta^{p\kappa}+N^p\right)
\end{align*}
with certain constant $K'_p$. Assuming that $\Delta\le (2K'_p)^{-1/(p\kappa)}$, we get
\begin{equation}\label{exnrp}
\ex{\norm{Y^{N,R}}_{0,t,\infty}^p} \le 2K'_{p} \left(\ex{\norm{Y^{N,R}}_{0,s,\infty}^p} +N^p\right).
\end{equation}
Finally, put $\Delta = \min\set{(4K'N)^{-1/\gamma},(2K'_{p})^{-1/(p\kappa)}}$. Splitting the segment $[0,T]$ into $[T/\Delta]+1$ parts of length at most $\Delta$, we obtain from the estimate \eqref{exnrp} that
\begin{equation*}
\ex{\norm{Y^{N,R}}_{0,T,\infty}^p} \le (2K'_{p}+1)^{T/\Delta + 1} \left(\abs{\eta(0)}^p +N^p\right)\le 
C\exp\set{C N^{1/\gamma}}.
\end{equation*}
Letting  $R\to\infty$ and using the Fatou lemma, we get
$$
\ex{\norm{X}_{0,T,\infty}^p\ind{\norm{Z}_{0,T,\gamma}\le N}} \le 
K'_p\exp\set{K'_p N^{1/\gamma}}
$$
with some constant $K'_p$.
Denote $\xi = \norm{X}_{0,T,\infty}^p$, $\eta = \norm{Z}_{0,T,\gamma}$ and write
\begin{align*}
\left(\ex{\xi^p}\right)^2 &\le \ex{\exp\set{2K'_{2p}\eta^{1/\gamma}}}\ex{\xi^{2p}\exp\set{-2K'_{2p} \eta^{1/\gamma}}}\\&\le C \sum_{n=1}^\infty \ex{\xi^{2p}\exp\set{-2K'_{2p} \eta^{1/\gamma}}\ind{\eta\in[n-1,n)}}\\
&\le C\sum_{n=1}^{\infty}\exp\set{-2K'_{2p}(n-1)^{1/\gamma}} \ex{\xi^{2p}\ind{\eta\in[n-1,n)}}\\
&\le C\sum_{n=1}^{\infty}\exp\set{-2K_{2p}'(n-1)^{1/\gamma}}\exp\set{K_{2p}'n^{1/\gamma}}<\infty,
\end{align*}
where the last constant depends on the parameters specified. The proof is now complete.
\end{proof}

We will impose the following assumptions concerning the convergence.
\begin{enumerate}[С1.]
\item
Convergence of the coefficients:
For all $\psi\in\CC$, $t\in[0,T]$,
\begin{gather*}
a^n(t,\psi)\to a^0(t,\psi),\quad b^n(t,\psi)\to b^0(t,\psi),\quad c^n(t,\psi)\to c^0(t,\psi),\quad n\to\infty.
\end{gather*}
\item Convergence of the initial conditions:
$$
\norm{\eta^n-\eta^0}_{\CC}\to 0, \quad n\to\infty.
$$
\end{enumerate}

\begin{theorem}\label{thm:main}
Under assumptions \textup{H1--H5} and \textup{C1--C2}, the following convergence in probability takes place:
\begin{equation*}
\norm{X^n-X^0}_{\infty,T}\overset{\pr}{\longrightarrow} 0,\quad n\to\infty.
\end{equation*}
\end{theorem}

\begin{proof}Obviously, we can assume without loss of generality that $Z(0)=0$.  Let for $N\ge 1$, $x\in\R^d$  
$$h_N(x) = \begin{cases}
x, &\abs{x}\le N,\\
N\frac{x}{\abs{x}}, & \abs{x}>N.
\end{cases}$$ 
Define the following sequence of smooth approximations of $Z$:
 $$
 Z^{ N}(t) = N \int_{(t-1/N)\vee 0}^t h_N(Z(s))\, ds,\quad   N\ge 1.
 $$
Consider auxiliary stochastic differential equations
\begin{equation}\label{eqAu1}
\begin{aligned}
X^{n, N}(t)& = \eta^n(0) + \int_0^t a^n(s, X^{n, N}_s)ds +  \int_0^t b^n(s, X^{n, N}_s)dW(s) +  \int_0^t c^n(s, X^{n, N}_s)d Z^{N}(s),\ t\in[0,T];\\
X^{n, N}(t)& = \eta^n(t),\ t\in[-r,0].
\end{aligned}\end{equation}
Since $Z^N$ is absolutely continuous, $$dZ^{N}(t) = N
\left(h_N(Z(t)) - h_N\big(Z((t-1/N)\vee 0)\big)\right)dt =: \dot Z^{ N}(t)\,dt $$  
we can rewrite the equation \eqref{eqAu1} as an It\^o delay differential  equation with random coefficients
\begin{equation}
\begin{aligned}
X^{n, N}(t)& = \eta^n(0) + \int_0^t f^{n,N}(s, X^{n, N}_s)ds +  \int_0^t b^n(s, X^{n, N}_s)dW(s),\ t\in[0,T];\\
X^{n, N}(t)& = \eta^n(t),\ t\in[-r,0],
\end{aligned}
\end{equation}
where $f^{n,N}(t, \psi) = a^n(t,\psi) + c^n(t,\psi) \dot{Z}^N(t)$, $t\in[0,T]$, $\psi\in\CC$. Obviously, the coefficients of these equations satisfy assumptions of Theorem~\ref{thm:itosddeconv} from Appendix~A. 
Therefore, 
\begin{equation}\label{XnN-Xn}
\norm{X^{n,N}-X^{0,N}}_{\infty,T}\overset{\pr}{\longrightarrow} 0,\quad n\to\infty.
\end{equation}
Now write for $\eps>0$
\begin{gather*}
\pr\left(\norm{X^n-X^0}_{\infty,T}>\eps\right)\le 
\pr\left(\norm{X^{n}-X^{n,N}}_{\infty,T}>\eps/3\right) \\
+ 
\pr\left(\norm{X^{n,N}-X^{0,N}}_{\infty,T}>\eps/3\right) + 
\pr\left(\norm{X^{0,N}-X^0}_{\infty,T}>\eps/3\right).
\end{gather*}
Hence,
\begin{gather*}
\limsup_{n\to\infty}\pr\left(\norm{X^n-X^0}_{\infty,T}>\eps\right)\le  2
\sup_{n\ge 0}\pr\left(\norm{X^{n}-X^{n,N}}_{\infty,T}>\eps/3\right)=:2\sup_{n\ge 0}P(A_{n,N,\eps}).
\end{gather*}
We need to show that $\sup_{n\ge 0}P(A_{n,N,\eps})\to 0$ as $N\to\infty$. To this end, write for any $M>0$, $R>0$
\begin{gather*}
\pr\left(A_{n,N,\eps}\right)\le 
\pr\left(A_{n,N,\eps}, \norm{X^{n,N}}_{T}\le R,\norm{X^{n}}_{T}\le R, \norm{Z^N}_{0;T}\le M, \norm{Z}_{0;T}\le M\right)\\
+ \pr\left(\norm{X^{n,N}}_{T}> R, \norm{Z^N}_{0;T}\le M\right) + \pr\left(\norm{X^{n}}_{T}> R, \norm{Z}_{0;T}\le M\right)\\
+ \pr\left(\norm{Z^N}_{0;T}> M\right) + \pr\left(\norm{Z}_{0;T}> M\right).
\end{gather*}
Since $\norm{Z^N-Z}_{0;T}\to 0$, $N\to\infty$ a.s. (it can be proved similarly to Lemma 2.1 in \cite{mbfbm-limit}), it follows easily from Lemma~\ref{prop-quasicontract} that 
$$
\sup_{n\ge 0}
\pr\left(A_{n,N,\eps}, \norm{X^{n,N}}_{T}\le R,\norm{X^{n}}_{T}\le R, \norm{Z^N}_{0;T}\le M, \norm{Z}_{0;T}\le M\right)\to 0,\ N\to\infty.
$$
Further, from Lemma~\ref{prop-apriormoment} we have with the help of Chebyshev inequality that $$\sup_{n\ge 0}\left(\sup_{N\ge 1} \pr\left(\norm{X^{n,N}}_{T}> R, \norm{Z^N}_{0;T}\le M\right)+\pr\left(\norm{X^{n}}_{T}> R, \norm{Z}_{0;T}\le M\right)\right)\to 0,\ R\to\infty.$$ 
Finally, 
$$
\sup_{N\ge 1}\pr\left(\norm{Z^N}_{0;T}> M\right) + \pr\left(\norm{Z}_{0;T}> M\right)\to 0,\ M\to\infty.
$$
Thus, we arrive to 
$$
\sup_{n\ge 0}P(A_{n,N,\eps})\to 0,\ N\to\infty,
$$
as required.
\end{proof}

\begin{corollary}
Assume that coefficients of \eqref{sdde-coord} satisfy assumptions {\rm H1--H5} and 
 $$\ex{\exp\{\norm{Z}^{1/\gamma}_{0,T,\gamma}\}}<\infty.$$
Then  for all $p\geq 1$
$$
\Ex{\norm{X^n-X^0}_{0,T,\infty}^p}\to 0, n\to \infty. 
$$
\end{corollary}
\begin{proof}
Theorem~\ref{thm:main} implies the boundedness of sequence $\{\norm{X^n-X^0}_{0,T,\infty}^q,n\ge 1\}$ in $L^q(\Omega)$ for all $q>p$. Therefore, the statement of the corollary follows from {Theorem \ref{thm:main}} thanks to the uniform integrability. 
\end{proof}
\section{Applications}\label{sec:applic}

\subsection{Vanishing delay}\label{subsec:vanish}
Let, as above, the process $\set{Z(t),t\ge 0}$ be an $\F$-adapted process in $\R^l$ with $\gamma$-H\"older continuous paths, $\gamma>1/2$, $\set{W(t),t\ge 1}$ be a standard $\F$-Wiener process in $\R^m$. 

Consider the following sequence of equations with delay in $\R^d$.
\begin{equation}\label{vanishingdelayeq}
\begin{aligned}
X^n(t) &= \eta(0) + \int_0^t a^V(s,X^n(s),X^n(s-\tau_n))ds + \int_0^t b^V(s,X^n(s),X^n(s-\tau_n))dW(s)\\&\qquad + \int_0^t c^V(s,X^n(s),X^n(s-\tau_n))dZ(s),\\
X^n(t) &= \eta(t), t\in[-\tau_n,0].
\end{aligned}
\end{equation}
Here $a^V\colon [0,T]\times \R^{2d}\to\R^d$, $b_i^V\colon [0,T]\times \R^{2d}\to\R^{d}$, $i=1,\dots,m$,  $c_j^V\colon  [0,T]\times \R^{2d}\to\R^{d}$, $j=1,\dots,l$, $\eta\colon[0,T]\to \R^d$ are measurable functions;  $\{\tau_n,n\ge 1\}$ is a   sequence of positive numbers with $\tau_n<r$.

The assumptions about the coefficients are similar to H1--H5.
\begin{enumerate}[HV1.]
\item For all $t\in[0,T]$, $x,y\in \R^d$
$$
\abs{a^V(t,x,y)}+ \abs{b^V(t,x,y)} + \abs{c^V(t,x,y)}\le K(1+\abs{x}+\abs{y}).
$$
\item For all $t\in[0,T]$, $x,y\in \R^d$ there exist bounded derivatives $\partial_x c^V(t,x,y)$, $\partial_y c^V(t,x,y)$:
$$
\abs{\partial_x c^V(t,x,y)} + \abs{\partial_y c^V(t,x,y)}\le K.
$$
\item For all $t\in[0,T]$, $R>1$, and $x_1,x_2,y_1,y_2\in \R^d$ with $\abs{x_i}\le R$, $\abs{y_i}\le R$, $i=1,2$,
\begin{gather*}
\abs{a^V(t,x_1,y_1)-a^V(t,x_2,y_2)}+ \abs{b^V(t,x_1,y_1)-b^V(t,x_2,y_2)}\\
 + \abs{\partial_x c^V(t,x_1,y_1)-\partial_x c^V(t,x_2,y_2)}  + \abs{\partial_y c^V(t,x_1,y_1)-\partial_y c^V(t,x_2,y_2)}\\\le K_R(\abs{x_1-x_2}+\abs{y_1-y_2}).
\end{gather*}
\item There exists $\beta\in(1-\gamma,1)$ such that for all $s,t\in[0,T]$, $x,y\in \R^d$
\begin{gather*}
\abs{c^V(t_1,x,y)-c^V(t_2,x,y)}\le K\abs{t_1-t_2}^\beta(1+\abs{x}+\abs{y}),\\
\abs{\partial_x c^V(t_1,x,y)-\partial_x c^V(t_2,x,y)} + \abs{\partial_y c^V(t_1,x,y)-\partial_y c^V(t_2,x,y)}\le K\abs{t_1-t_2}^\beta.
\end{gather*}
\item The initial condition  $\eta$ is a H\"older continuous function: for some  $\theta\in(1-\gamma, 1/2)$ and for all $t_1,t_2\in[0,T]$
$$
\abs{\eta(t_1)-\eta(t_2)}\le K\abs{t_1-t_2}^\theta.
$$
\end{enumerate}
From  Theorem~\ref{thm:main} we deduce the following  result about vanishing delay convergence.
\begin{theorem}
Assume that the coefficients of equations \eqref{vanishingdelayeq} satisfy \textup{HV1--HV5}, and $\tau_n\to 0$, $n\to\infty$. Then we have the following uniform convergence:
$$
\norm{X^n-X}_{\infty,T}\overset{\pr}{\longrightarrow} 0,\ n\to\infty,
$$
to the solution $X$ of equation
\begin{gather*}
X(t) = \eta(0) + \int_0^t a^V(s,X(s),X(s))ds + \int_0^t b^V(s,X(s),X(s))dW(s) + \int_0^t c^V(s,X(s),X(s))dZ(s).
\end{gather*}
\end{theorem}
\begin{proof}
Set $\tau_0 = 0$ and define for $n\ge 0$ the following sequence of functions   $a^n(s,\psi) = a^V(s,\psi(0),\psi(-\tau_n))$,  $b^n_i(s,\psi) = b^V_i(s,\psi(0),\psi(-\tau_n))$, $i=1,\dots, m$, $c^n_j(s,\psi) = c^V_j(s,\psi(0),\psi(-\tau_n))$, $j=1,\dots,l$, where $t\in[0,T]$, $\psi\in \CC$. These coefficients are easily seen to satisfy assumptions H1--H4. Moreover, since the $a^V$, $b^V$, $c^V$ are continuous, we have  for any $t\in [0,T]$, $\psi\in\CC$ the convergence $a^n(t,\psi)\to a^0(t,\psi)$, $b^n(t,\psi)\to b^0(t,\psi)$, $c^n(t,\psi)\to c^0(t,\psi)$ as $n\to\infty$. The proof is finished by observing that for such coefficients the solutions to \eqref{sdde-coord} coincide with those of \eqref{vanishingdelayeq} and applying Theorem~\ref{thm:main}.
\end{proof}

\subsection{Euler approximations}\label{subsec:euler}

Consider now a standard mixed stochastic differential equation
\begin{equation}\label{eulereq}
X(t) = X(0) + \int_0^t a^E(s,X(s))ds + \int_0^t b^E(s,X(s))dW(s) + \int_0^t c^E(s,X(s))dZ(s).
\end{equation}

Here $a^E\colon [0,T]\times \R^{d}\to\R^d$, $b^E_i\colon [0,T]\times \R^{d}\to\R^{d}$, $i=1,\dots,k$,  $c^E_j\colon  [0,T]\times \R^{d}\to\R^{d}$, $j=1,\dots,l$, satisfy the following assumptions.
\begin{enumerate}[HE1.]
\item For all $t\in[0,T]$, $x\in \R^d$
$$
\abs{a^E(t,x)}+ \abs{b^E(t,x)} + \abs{c^E(t,x)}\le K(1+\abs{x}).
$$
\item For all $t\in[0,T]$, $x\in \R^d$ there exists a bounded derivative $\partial_x c^E(t,x)$:
$$
\abs{\partial_x c^E(t,x)} \le K.
$$
\item For all $t\in[0,T]$, $R>1$, and $x_1,x_2\in \R^d$ with $\abs{x_1}\le R$, $\abs{x_2}\le R$
\begin{gather*}
\abs{a^E(t,x_1)-a^E(t,x_2)}+ \abs{b^E(t,x_1)-b^E(t,x_2)} + \abs{\partial_x c^E(t,x_1)-\partial_x c^E(t,x_2)}\le K_R\abs{x_1-x_2}.
\end{gather*}
\item There exists $\beta\in(1-\gamma,1)$ such that for all $s,t\in[0,T]$, $x,y\in \R^d$
\begin{gather*}
\abs{c^E(t_1,x)-c^E(t_2,x)}\le K\abs{t_1-t_2}^\beta(1+\abs{x}),\\
\abs{\partial_x c^E(t_1,x)-\partial_x c^E(t_2,x)}\le K\abs{t_1-t_2}^\beta.
\end{gather*}
\end{enumerate}

Euler approximations for the solution of \eqref{eulereq} are constructed as follows. For $n\ge 1$ define $\delta=T/n$ and consider a uniform partition of $[0,T]$: $t_k^n = k\delta$, $k=0,1,\dots,n$. 
Define recursively
\begin{align*}
&X^n(0) = X(0),\\
&X^n(t^n_{k+1}) = X^n(t^n_{k}) + a^E(t^n_{k},X^n(t^n_{k}))\delta + b^E(t^n_{k},X^n(t^n_{k}))\left(W(t_{k+1}^n)-W(t_k^n)\right)\\&\qquad\qquad\qquad{} + c^E(t^n_{k},X^n(t^n_{k}))\left(Z(t_{k+1}^n)-Z(t_k^n)\right), k\ge 0.
\end{align*}
Denoting $t^n(s) = \max\set{t^n_k:t^n_k\le s}$, we can interpolate the values of approximations with
\begin{gather*}
X^n(t) = X(0) + \int_0^t a^E\big(t^n(s),X^n(t^n(s))\big)ds + \int_0^t b^E\big(t^n(s),X^n(t^n(s))\big)dW(s)\\
+ \int_0^t c^E\big(t^n(s),X^n(t^n(s))\big)dZ(s).
\end{gather*}
This can be considered as an equation with delay, however, its coefficient $c^n(s,\psi):= c^E(t^n(s),\psi(t^n(s)))$ does not satisfy the assumption H4, so one needs to prove analogues of Lemmas \ref{prop-apriormoment} and \ref{prop-quasicontract}. This can be done with a slight modification of corresponding arguments in \cite{mbfbm-delay};	 we skip the proof as this is not our main interest here.  Thus, we have the following result.
\begin{theorem}\label{thm:eulerconv}
Assume that the coefficients of equations \eqref{eulereq} satisfy \textup{HE1--HE4}. Then we have the following uniform convergence of Euler approximations:
$$
\sup_{t\in[0,T]}\abs{X^n(t)-X(t)}\overset{\pr}{\longrightarrow} 0,\ n\to\infty.
$$
\end{theorem}
\begin{remark}
While establishing the convergence of approximations, Theorem~\ref{thm:eulerconv} tells nothing about the rate of convergence. In \cite{mbfbm-euler}, the rate of convergence was established for a one-dimensional equation driven with $Z=B^H$ under more restrictive assumptions on the coefficients.
\end{remark}

\appendix
\section{Limit theorem for It\^o delay equations}\label{sec:itosdde}
\renewcommand{\thesection}{\Alph{section}}
Here we prove a convergence result for It\^o SDDEs.  Consider a sequence of stochastic delay differential equations in $\R^d$:
\begin{equation}\label{osdde-coord}
Y^n(t,\omega) = \theta^n(0,\omega) + \int_{0}^{t} f^n(s,Y^n_s,\omega)ds + \sum_{i=1}^{m} \int_0^t g^n_i(s,Y^n_s,\omega)dW_i(s),
\end{equation}
or, shortly,
\begin{equation*} %\label{osdde-coord}
Y^n(t) = \theta^n(0) + \int_{0}^{t}f^n(s,Y^n_s)ds + \int_0^t g^n(s,Y^n_s)dW(s),
\end{equation*}
with $\F_0$-measurable initial conditions $Y^n(t,\omega) = \theta^n(t,\omega)$, $t\in[-r,0]$. 
These equations are similar to  \eqref{sdde-coord}, but they do not contain a part with the process $Z$. Another difference is that the coefficients of \eqref{osdde-coord} are random. We will impose the following assumptions on them.
\begin{enumerate}[H{A}1.]
\item For all $\psi\in\CC$, $t\in[0,T]$ \ $f^n(t,\psi)$ and $g^n(t,\psi)$ are $\F_t$-measurable.
\item
For all $\psi\in\CC$, $t\in[0,T]$ and a.a.\ $\omega\in\Omega$
\begin{gather*}
|f^n(t,\psi,\omega)|+|g^n(t,\psi,\omega)|\leq K(1+\norm{\psi}_{\CC}).
\end{gather*}
\item The functions $f^n$ and $g^n$ are locally Lipschitz continuous in $\psi$: for all $R>1$, $t\in[0,T]$, a.a.\ $\omega\in\Omega$, and any $\psi_1,\psi_2\in \CC$ with $\norm{\psi_1}_\CC\le R$,  $\norm{\psi_2}_\CC\le R$,
$$|f^n(t,\psi_1,\omega)-f^n(t,\psi_2,\omega)|+|g^n(t,\psi_1,\omega)-g^n(t,\psi_2,\omega)|\leq K_R\norm{\psi_1-\psi_2}_{\CC}.$$
\end{enumerate}
The unique solvability can be shown by slight modification of arguments in\cite[Theorem I.2]{mohammed} and \cite[Chapter 5, Theorem 2.5]{mao}. Moreover, there is a uniform integrability of solutions, which we state below.
\begin{proposition}\label{sddeA-boundedmoments} Under assumptions \textup{HA1--HA2}, for any $p\ge 2$ 
\begin{equation}
\ex{\norm{Y^n}_{\infty,T}^p\,\Big|\, \F_0}\le C_p(1+\norm{\theta^n}^p_{\CC}).
\end{equation}
\end{proposition}
\begin{proof}
Let $M^n(t) =\sup_{s\in[0,t]}\abs{Y(s)}^p$. For a fixed $R>0$, denote $\ind{t} = \ind{M^n(t)\le R}$. Also abbreviate $\exo{{}\cdot{}} = \ex{{}\cdot{}\mid \F_0}$. 

Estimate 
\begin{gather*}
M^n(t) \le C_p \left(\abs{\theta(0)}^p + \sup_{s\in[0,t]}\abs{\int_{0}^{s} f^n(u,Y^n_u)du}^p + 
\sup_{s\in[0,t]}\abs{\int_{0}^{s} g^n(u,Y^n_u)dW(u)}^p\right)\\
\le  C_p \left(\abs{\theta(0)}^p + \int_{0}^{t} \abs{f^n(s,Y^n_s)}^p ds + 
\sup_{s\in[0,t]}\abs{\int_{0}^{s} g^n(u,Y^n_u)dW(u)}^p\right)
\end{gather*}
Therefore, 
\begin{gather*}
\exo{M^n(t) \ind{t}}\le C_p \left(\abs{\theta(0)}^p + I^a_t + I^b_t\right),
\end{gather*}
where
\begin{align*}
&I^a_t = \int_{0}^{t} \exo{\abs{f^n(s,Y^n_s)}^p\ind{t}} ds,\\
&I^b_t = \exo{\sup_{s\in[0,t]}\abs{\int_{0}^{s} g^n(u,Y^n_u)dW(u)}^p\ind{t}}.
\end{align*}
Since the events $\set{M^n_t\le R}$ are decreasing in $t$, we can estimate
\begin{gather*}
I_t^a \le \int_{0}^{t}\exo{\abs{f^n(s,Y^n_s)}^p\ind{s}}ds\le C_p \int_0^t \exo{\left(1+\norm{Y_s^n}_{\CC}^p\right)\ind{s}}ds\\
\le C_p \int_0^t \exo{\big(1+ \norm{\theta^n}^p_{\CC} + M^n(s) \big)\ind{s}} ds\le C_p\left(1+\norm{\theta^n}^p_{\CC} + \int_0^t \exo{M^n(s)\ind{s}}ds\right).
\end{gather*}
Further, with the help of the Burkholder--Gundy--Davis inequality, we obtain
\begin{gather*}
I_t^b \le C_p\exo{\sup_{s\in[0,t]}\abs{\int_0^s g^n(u,Y^n_u)\ind{s}dW(u)}^p}\le C_p 
\exo{\left(\int_0^t g^n(s,Y^n)^2\ind{s}ds\right)^{p/2}}\\
\le C_p \int_0^t \exo{\abs{g^n(s,Y^n_s)}^p\ind{s}}ds\le C_p\left(1+\norm{\theta^n}^p_{\CC} + \int_0^t \exo{M^n(s)\ind{s}}ds\right),
\end{gather*}
where the last inequality is obtained the same way as for $I^a_t$.

Consequently,
\begin{gather*}
\exo{M^n(t)\ind{t}}\le C_p\left(1+\norm{\theta^n}^p_{\CC} + \int_0^t \exo{M^n(s)\ind{s}}ds\right).
\end{gather*}
By the Gronwall lemma, 
$$
\exo{M^n(t)\ind{t}}\le C_p\left(1+\norm{\theta^n}^p_{\CC}\right).
$$
Now letting $R\to\infty$ and using the Fatou lemma, we arrive at the statement.
\end{proof}

Concerning the convergence, we will assume the following.
\begin{enumerate}[C{A}1.]
\item
Pointwise convergence of the coefficients in probability: for all $\psi\in\CC$, $t\in[0,T]$ 
\begin{gather*}
f^n(t,\psi)\overset{\pr}{\longrightarrow} f^0(t,\psi),\quad g^n(t,\psi)\overset{\pr}{\longrightarrow} g^0(t,\psi), \quad n\to \infty.
\end{gather*}
\item Convergence of initial conditions in probability: 
$$ \norm{\theta^n - \theta^0}_{\CC}\overset{\pr}{\longrightarrow} 0,\quad n\to\infty.$$
\end{enumerate}

\begin{theorem}\label{thm:itosddeconv}
Under assumptions \textup{HA1--HA3} and \textup{CA1--CA2}, the following convergence in probability takes place:
\begin{equation*}
\norm{Y^n-Y^0}_{\infty,T}\overset{\pr}{\longrightarrow} 0,\quad n\to\infty.
\end{equation*}
Moreover, if additionally for some $p\ge 2$ \ $\ex{\norm{\theta^n-\theta^0}_{\CC}^p}\to 0$, $n\to\infty$,  then 
\begin{equation*}
\ex{\norm{Y^n-Y^0}_{\infty,T}^p}\to 0,\quad n\to\infty.
\end{equation*}
\end{theorem}
\begin{proof}
It is enough to prove that any subsequence of $\norm{Y^n-Y^0}_{\infty,T}^p$ contains a subsequence converging to zero in probability. Therefore, it can be assumed without loss of generality that $\norm{\theta^n-\theta^0}_{\CC}\to0$, $n\to\infty$, a.s.

Denote for some $p\ge 2$ \ $\Delta^n_t = \norm{Y^n-Y^0}_{\infty,t}^p$ and abbreviate, as in Proposition~\ref{sddeA-boundedmoments}, $\exo{{}\cdot{}} = \ex{{}\cdot{}\mid \F_0}$.
Write 
\begin{gather*}
\exo{\Delta_t^n}\le 
C_p\Bigg(\abs{\theta^n(0)-\theta^0(0)}^p  + \exo{\sup_{s\in[0,t]}\abs{\int_0^s \left(f^n(u,Y^n_u)-f^0(u,Y^0_u)\right)du}^p}\\
+ \exo{\sup_{s\in[0,t]}\abs{\int_0^s \left(g^n(u,Y^n_u)-g^0(u,Y^0_u)\right)dW(u)}^p}\Bigg)=: C_p \left(\abs{\theta^n(0)-\theta^0(0)}^p + J^a_t + J^b_t\right).
\end{gather*}
Estimate separately
\begin{gather*}
J^a_t \le C_p\int_{0}^{t}\exo{\abs{f^n(s,Y^n_s)-f^0(s,Y^0_s)}^p} ds\\
\le C_p \int_{0}^{t}\exo{\abs{f^n(s,Y^n_s)-f^n(s,Y^0_s)}^p + \abs{f^n(s,Y^0_s)-f^0(s,Y^0_s)}^p}  ds.
\end{gather*}
Now
\begin{gather*}
\int_{0}^{t}\exo{\abs{f^n(s,Y^n_s)-f^n(s,Y^0_s)}^p}  ds\le C_p \int_{0}^{t}\exo{\norm{Y^n_s-Y^0_s}_{\CC}^p}  ds\\
\le C_p \int_{0}^{t}\exo{\norm{\theta^n-\theta^0}^p_{\CC}+\Delta_s^n}ds\le C_p\left(\norm{\theta^n-\theta^0}^p_{\CC} + \int_{0}^{t}\exo{\Delta_s^n}ds\right).
\end{gather*}
Thus,
\begin{gather*}
J^a_t \le C_p\left(\norm{\theta^n-\theta^0}^p_{\CC} + \int_{0}^{t}\exo{\Delta_s^n}ds + \int_{0}^{t}\exo{ \abs{f^n(s,Y^0_s)-f^0(s,Y^0_s)}^p}  ds\right).
\end{gather*}
A similar estimate for $J^b_t$ is obtained with the help of the Burkholder--Gundy--Davis inequality.
Therefore, we arrive to 
\begin{equation}\label{limsupdeltan}
\begin{gathered}
\exo{\Delta_t^n}\le 
C_p\Bigg(\norm{\theta^n-\theta^0}^p_{\CC} + \int_{0}^{t}\exo{\Delta_s^n}ds + \int_{0}^{t}\exo{ \abs{f^n(s,Y^0_s)-f^0(s,Y^0_s)}^p}  ds\Bigg).\end{gathered}
\end{equation}
Thanks to the linear growth assumption HA2 and to Proposition~\ref{sddeA-boundedmoments}, $$\exo{ \abs{f^n(s,Y^0_s)-f^0(s,Y^0_s)}^p}\le C_p (1+\norm{\theta^0}_{\CC}^{p}).$$ Hence, in view of the dominated convergence theorem and assumption CA1, the last term in \eqref{limsupdeltan} vanishes as $n\to\infty$. Also by Proposition~\ref{sddeA-boundedmoments}, $\limsup_{n\to\infty} \exo{\Delta_s^n}\le C_p (1+\limsup_{n\to\infty}\norm{\theta^n}^{p}_{\CC})<\infty$. Thus, taking $\limsup_{n\to\infty}$ in \eqref{limsupdeltan}, we obtain 
\begin{gather*}
\limsup_{n\to\infty}\exo{\Delta_t^n}\le 
C_p \int_{0}^{t}\limsup_{n\to\infty}\exo{\Delta_s^n}ds.
\end{gather*}
By  the Gronwall lemma, $\limsup_{n\to\infty}\exo{\Delta_t^n} = 0$, yielding the first statement of the theorem. The second statement is obtained by taking expectation in equation \eqref{limsupdeltan} and repeating the argument following that equation.
\end{proof}

\bibliographystyle{elsarticle-harv}
\bibliography{mbfbm-convergence}

\end{document}